\documentclass[11pt, notitlepage]{article}
\usepackage{amssymb,amsmath,comment}
\catcode`\@=11 \@addtoreset{equation}{section}
\def\thesection{\arabic{section}}

\def\theequation{\thesection.\arabic{equation}}
\catcode`\@=12
\usepackage{colortbl}%
\usepackage[notref,notcite]{showkeys}
\usepackage{a4wide}

\newcommand{\ds} {\displaystyle}
\newcommand{\e}{\epsilon}

\newcommand{\al} {\alpha}
\newcommand{\ba} {\beta}
\newcommand{\de} {\delta}
\newcommand{\ga} {\gamma}

\newcommand{\Om} {\Omega}
\newcommand{\ra} {\rightarrow}

\newcommand{\De} {\Delta}
\newcommand{\la} {\lambda}

\newcommand{\noi} {\noindent}
\newcommand{\na} {\nabla}

\newcommand{\mb} {\mathbb}
\newcommand{\mc} {\mathcal}
\newcommand{\lra} {\longrightarrow}
\newcommand{\ld} {\langle}
\newcommand{\rd} {\rangle}

\usepackage[all]{xy}
\catcode`\@=11
\def\theequation{\@arabic{\c@section}.\@arabic{\c@equation}}
\catcode`\@=12

\def\QED{\hfill {$\square$}\goodbreak \medskip}

\newtheorem{Theorem}{Theorem}[section]
\newtheorem{Lemma}[Theorem]{Lemma}

\newtheorem{Example}{Example}

\begin{document}
\vspace{0.01in}

\title
{A note on the eigenvalues of fractional Hardy-Sobolev operator with indefinite weight}
\author{
{\bf  Sarika Goyal\footnote{email: sarika1.iitd@gmail.com}} \\
{\small Department of Mathematics}, \\{\small Indian Institute of Technology Delhi}\\
{\small Hauz Khas}, {\small New Delhi-16, India}\\}

\date{}

\maketitle

\begin{abstract}
\noi In this article, we study the eigenvalue of nonlinear $p-$fractional Hardy operator
{\small \begin{align*}
 (-\De)_{p}^{\al}u  -\mu \frac{|u|^{p-2}u}{|x|^{p\al}} = \la V(x) |u|^{p-2}u \; \text{in}\;
\Om, \quad  u = 0 \; \mbox{in}\; \mb R^n \setminus\Om,
\end{align*}}
where $n>p\al$, $p\geq2$, $\al\in(0,1)$, $0\leq \mu <C_{n,\al,p}$ and $\Om$ is a domain in
$\mb R^n$ with Lipschitz boundary containing $0$. In particular, $\Om=\mb R^n$ is admitted. The weight function $V$ may change sign and may have singular points. We also show that the least positive eigenvalue is simple and it is unique associated to a non-negative eigenfunction.
Moreover, we proved that there exists a sequence of eigenvalues $\la_k \ra \infty$ as $k\ra\infty$.
\medskip

\noi \textbf{Key words:} Eigenvalue problem, fractional Hardy-Sobolev operator, indefinite weight.

\medskip

\noi \textit{2010 Mathematics Subject Classification:} 35A15, 35B33,
35H39

\end{abstract}

\bigskip
\vfill\eject

\section{Introduction}
In this article, we study the nonlinear eigenvalue problem of $p-$fractional Hardy operator
 {\small\begin{equation}\label{p2}
 (-\De)_{p}^{\al} -\mu \frac{|u|^{p-2}u}{|x|^{p\al}} = \la V(x) |u|^{p-2}u \; \text{in}\;
\Om, \quad  u = 0 \; \mbox{in}\; \mb R^n \setminus\Om,
\end{equation}}
where $n>p\al$, $p\geq2$, $\al\in(0,1)$, $0\leq \mu <C_{n,\al,p}$, where $C_{n,\alpha,p}$ is a constant defined later and $\Om$ is a domain in
$\mb R^n$ with Lipschitz boundary containing $0$. We observe that $\Om$ may be unbounded, and in particular $\Om=\mb R^n$.  Here, $(-\De)^{\al}_{p}$ is the fractional Laplacian operator
defined as
\begin{equation*}
(-\De)_{p}^{\al} u(x)= -2 \lim_{\e\ra 0} \int_{\mb
R^n\setminus B_{\e}(x)}\frac{|u(x)-u(y)|^{p-2}(u(x)-u(y))}{|y|^{n+2\al}} dy \;\text{for all} \;
x\in \mb R^n.
\end{equation*}
We assume the following condition on $V$:
\begin{enumerate}
\item[$(A_p)$] $V\in L^{1}_{loc}(\Om)$, $V^{+}=V_1+V_2\ne 0$ with $V_1\in L^{\frac{n}{p\al}}(\Om)$ and $V_2$ is such that
\begin{align*}
\lim_{x\ra y, x\in\Om} |x-y|^{p\al} V_2(x)=0\; \mbox{for all}\; y\in \overline{\Om}, \;\mbox{and}\; \lim_{|x|\ra\infty, x\in\Om} |x|^{p\al} V_{2}(x)=0.
\end{align*}
\end{enumerate}
The eigenvalue problem \eqref{p2} is motivated by the following Sharp fractional Hardy inequality:
\begin{Theorem}(see \cite{fs})
 Let $n\geq 1$ and $0<\al<1$. Then for all $u\in \dot{W}_{p}^{\al}(\mb R^n)$ in case $1\leq p <n/\al$, and for all $u\in \dot{W}^{\al}_{p}(\mb R^n\setminus \{0\})$ in case $p>n/\al$, we have
\begin{align}\label{s1}
\ds{\int\int}_{\mb R^n\times \mb R^n} \frac{|u(x)-u(y)|^p}{|x-y|^{n+p\al}}dxdy \geq C_{n,\al,p}\int_{\mb R^n}\frac{|u(x)|^p}{|x|^{p\al}}dx
\end{align}
with
\[C_{n,\al,p}:= 2\int_{0}^{1} r^{p\al-1}|1-r^{(n-p\al)/p}|^p \Phi_{n,\al,p}(r) dr, \]
and
\[\Phi_{n,\al,p}(r):= |{\mb S}^{n-2}| \int_{-1}^{1}\frac{(1-t^2)^{\frac{n-3}{2}}dt }{(1-2rt+r^2)^{\frac{n+p\al}{2}}},\; n\geq 2,\]
\[\Phi_{1,\al,p}(r):= \left(\frac{1}{(1-r)^{1+p\al}}+\frac{1}{(1+r)^{1+p\al}}\right),\; n=1. \]
The constant $C_{n,\al,p}$ is optimal. If $p=1,$ equality in \eqref{s1} holds iff $u$ is proportional to a symmetric decreasing function. If $p>1$, the inequality in \eqref{s1} is strict for any function $0\not\equiv u \in \dot{W}_{p}^{\al}(\mb R^n)$ or $\dot{W}_{p}^{\al}(\mb R^n\setminus \{0\})$ respectively.
\end{Theorem}
\noi The fractional power of Laplacian is the infinitesimal generator of L\'{e}vy stable diffusion process and arise in anomalous diffusions in
plasma, population dynamics, geophysical fluid dynamics, flames propagation, chemical reactions in liquids and American options in finance. For more
details, one can see \cite{da, gm} and reference therein. Recently the fractional elliptic equation attracts a lot of interest in nonlinear analysis
such as in \cite{cs, mp,var,bn1,bn}. Caffarelli and Silvestre \cite{cs} gave a new formulation of fractional Laplacian through Dirichlet-Neumann maps.
This is commonly used in the literature since it allows us to write a nonlocal problem to a local problem which allow us to use the variational
methods to study the existence and uniqueness.

\noi When $\al=1$, the $p$-fractional Laplacian operator becomes the $p$-Laplace operator. A lot of work has been done in case of $p-$Laplacian see \cite{as,ss,vs,s2,sw,s,sre} and references their in. For $\mu=0$ and $\al=1$, the problem \eqref{p2} is studied by Szulkin and Willem in \cite{sw}. Also, for $\mu\ne0$, the problem \eqref{p2} is studied by Sreenadh in \cite{sre}, which is motivated by the following $p-$Laplace Hardy-Sobolev inequality,
\[\int_{\Om} |\na u|^p dx \geq \left(\frac{n-p}{p}\right)^p \int_{\Om}\frac{|u|^p}{|x|^p} dx,\]
for all $1<p<n$ and for all $u\in W_{0}^{1,p}(\Om)$.

\noi In case of fractional operator, $\mu=0$, and $V\equiv 1$, the eigenvalue problem \eqref{p2} is studied
by Servadei and Valdinoci \cite{var} (for $p=2$), Lindgren and Lindqvist \cite{ll} (for the case $p\geq 2$) and by Franzina and Palatucci \cite{fp} (for any
$p > 1$). In \cite{ll}, a lot of attention is given to the asymptotic behavior of problem \eqref{p2} as $p\ra\infty$, while
in \cite{fp}, some regularity results for the eigenfunctions are proved. In \cite{as1}, the author studied the higher eigenvalues of $p-$fractional operator. Stability of the variational eigenvalues is studied in \cite{bps}. In \cite{pfr}, authors studied the first eigenvalue and properties of the $p$-fractional operator with a potential. In case of $\mu=0$, $V\in L^{\infty}(\Om)$, the properties of first eigenvalue are also studied in \cite{pq}.

\noi To the best of our knowledge, all the results in this paper are new even in the linear case. There is no work related to the eigenvalue of fractional hardy operator with indefinite weight. This work is motivated by the work of Szulkin and Willem \cite{sw} and Sreenadh \cite{sre}, which has been done for $p-$Laplacian operator. In this paper, we prove the existence of eigenvalues of fractional Hardy-Sobolev operator with singular type indefinite weight and analyze some properties of the associated eigenfunction. Moreover, we show the existence of a sequence $\{\la_k\}$ of eigenvalues such that $\la_k\ra \infty$ as $k\ra\infty$.

\noi The paper is organized as follows: In section 2, we give some preliminaries result. In section 3, we show the existence of first eigenvalue and proved some properties of first eigenvalue and corresponding eigenfunction. We give some examples in section 4. In section 5, we obtain a sequence of eigenvalues $\la_k$ such that $\la_k\ra\infty$ as $k\ra\infty$.

We shall throughout use the function space $X_0$ with the norm
$\|.\|$ and we use the standard $L^{p}(\Om)$ space whose norms
are denoted by $\|u\|_{p}$.

 \section{Prelimenaries}
 In this section, we first define the function space and prove some properties which are useful to find the solution of the the problem \eqref{p2}. For this we define $W^{\al,p}(\Om)$, the usual fractional Sobolev
space $W^{\al,p}(\Om):= \left\{u\in L^{p}(\Om); \frac{(u(x)-u(y))}{|x-y|^{\frac{n}{p}+\al}}\in L^{p}(\Om\times\Om)\right\}$ endowed with the norm
\begin{align}\label{2}
\|u\|_{W^{\al,p}(\Om)}=\|u\|_{L^p}+ \left(\int_{\Om\times\Om}
\frac{|u(x)-u(y)|^{p}}{|x-y|^{n+p\al}}dxdy \right)^{\frac 1p}.
\end{align}
To study fractional Sobolev space in details we refer \cite{hic}.

 \noi For $n\geq 1$ and $0<\al<1$, the fractional Sobolev space $\dot{W}_{p}^{\al}(\mb R^n)$ and $\dot{W}_{p}^{\al}(\mb R^n\setminus\{0\})$  defined as the completion of $C_{c}^{\infty}(\mb R^n)$ for $1\leq p< n/\al$ and $C_{c}^{\infty}(\mb R^n\setminus \{0\})$ for $p> n/\al$ respectively, with respect to the norm  \[{\int\int}_{\mb R^n\times \mb R^n} \frac{|u(x)-u(y)|^p}{|x-y|^{n+p\al}} dxdy.\]
Then we define
 \[ X_0 = \mbox{completion of}\{u\in C_{c}^{\infty}(\mb R^n) : u = 0 \;\text{in}\; \mb R^n\setminus \Om\}\]
with respect to the norm
\[\|u\|_{X _0}=\left( \int_{\mb R^{2n}}\frac{|u(x)- u(y)|^{p}}{|x-y|^{n+p\al}}dx
dy\right)^{\frac1p}= \left(\int_{Q}\frac{|u(x)- u(y)|^{p}}{|x-y|^{n+p\al}}dx
dy\right)^{\frac1p},\]
where  $Q=\mb R^{2n}\setminus(\mc C\Om\times \mc C\Om)$ and
 $\mc C\Om := \mb R^n\setminus\Om$. Then $X_0$ is a reflexive Banach space. Note that the norm
$\|.\|_{X_0}$ involves the interaction between $\Om$ and $\mb
R^n\setminus\Om$. This type of functional setting is introduced by Servadei and Valdinoci for $p=2$ in \cite{mp}.

\noi Now we state some result which are used to prove the main result:
\begin{Lemma}(Brezis-Lieb Lemma \cite{bl})\label{bl1}
 Suppose $f_k\ra f$ a.e and $\|f_k\|_p\leq C<\infty$ for all $k$ and for some $1<p<\infty.$ Then
 \[\lim_{k\ra\infty}\{\|f_k\|_{p}^p-\|f_k-f\|_{p}^{p}\}=\|f\|_{p}^{p}.\]
\end{Lemma}
 In order to prove simplicity of the first eigenvalue, we need the following discrete Picone-type identity (see Lemma 6.2 in \cite{am}).
\begin{Theorem}\label{dp1}
Let $p\in (1,+\infty)$. For $u$, $v: \Om\subset \mb R^n \ra \mb R$ such that $u\geq 0$ and $v>0$, we have
\[L(u,v)\geq 0\;\mbox{in}\; \mb R^n\times \mb R^n\]
where
\[L(u,v)(x,y)= |u(x)-u(y)|^p -|v(x)-v(y)|^{p-2}(v(x)-v(y))\left(\frac{u(x)^p}{v(x)^{p-1}}-\frac{u(y)^p}{v(y)^{p-1}}\right).\]
The equality holds if and only if $u=kv$ a.e. for some constant $k$.
\end{Theorem}

\section{On the first eigenvalue of Hardy fractional $p-$Laplacian}
In this section, we prove the existence of first eigenvalue of the problem . We show that the first eigenvalue $\la_1$ is simple by  discrete Picone identity and eigenfunctions associated to $\la_1$ have definite sign in $\Om$. Finally, we also prove a result on the strict monotonicity of $\la_1$ with respect to both the domain and the weight.

First, we note the some properties of the $p-$fractional Hardy operator
\[L_{\mu}(u):= -2 \int_{\mb R^n}\frac{|u(x)- u(y)|^{p-2} (u(y)-u(x))}{|x-y|^{n+p\al}}dy  -\mu \frac{|u|^{p-2}u}{|x|^{p\al}}.\]
\begin{enumerate}
\item[$(i)$] $L_{\mu}(u)$ is a positive operator,\; {for all}\; $u\in X_0$\;{and}\;$0\leq\mu<C_{n,\al,p}$.
\end{enumerate}
Also, if
\[ R(u) :=  \left(\int_{\mb R^{2n}}\frac{|u(x)- u(y)|^{p}}{|x-y|^{n+p\al}}dx
dy - \mu \int_{\Om} \frac{|u|^p}{|x|^{p\al}} dx\right)^{\frac1p}.\]
Then, by the fractional Hardy inequality \eqref{s1}, for $0\leq\mu<C_{n,\al,p}$, we have
\begin{enumerate}
\item[$(i)$]  $R(u)$ is equivalent to $\|\cdot\|_{X_0}$ but it does not define the norm itself.
\end{enumerate}
We define the functional $J_{\mu}$ on $X_0$ as
\[J_{\mu}(u):= \int_{Q} \frac{|u(x)-u(y)|^p}{|x-y|^{n+p\al}}dxdy- \mu\int_{\Om} \frac{|u|^p}{|x|^{p\al}}dx.\]
Then $J$ is $C^{1}$ on $X_0$. Our goal is to study the eigenvalue problem and some properties (simplicity, isolatedness) of
\begin{align}\label{i1}
\la_1:= \inf\{J_{\mu}(u): u\in M:=\{u\in X_0: \int_{\Om} V|u|^p dx=1\}\}.
\end{align}

\noi In order to prove the existence of an eigenvalue, we need to prove the following Lemma:
\begin{Lemma}\label{v1}
Under the assumption $(A_p)$, the map $T: X_0 \ra \mb R$ defined as
\[T(u)= \int_{\Om} V^{+} |u|^p dx\] is weakly continuous.
\end{Lemma}

\begin{proof}
Step $1$. One can see the proof of Lemma 2.13 in \cite{wi} for $\al=1$ but for completeness, we give the detail. Firstly, we show that $\int_{\Om} V_1 |u|^p dx$ is weakly continuous. For this, suppose $u_k\rightharpoonup u$ weakly in $X_0$. Then by fractional Sobolev inequality, $u_k \ra u$ strongly in $L^{r}$ for $1\leq r<p_{\al}^*= \frac{np}{n-p\al}$ and $u_k\ra u$ a.e on $\Om$. Since $\{u_k\}$ is bounded in $L^{p_{\al}^*}$, $\{u_k^p\}$ is bounded in $L^{\frac{n}{n-p\al}}(\Om)$. Hence up to a subsequence, $u_k^p\rightharpoonup u^p$ weakly in $L^{\frac{n}{n-p\al}}$ and $V_1\in L^{\frac{n}{p\al}}(\Om)$ so $\int_{\Om} V_1 |u_k|^p dx \ra \int_{\Om} V_1 |u|^p dx$.

\noi Step $2.$ In order to prove that $\int_{\Om} V_2 |u|^p dx$ is weakly continuous. Let $u_k\rightharpoonup u$ weakly in $X_0$ and $\e>0$. Then by assumption, there exists $R>0$ such that if $x\in \Om$ and $|x|\geq R$, then $|x|^{p\al} V_2(x)\leq \e$. Define
\[\Om _1:= \Om\setminus B[0,R],\quad \Om_2:= \Om\setminus B(0,R), \quad c:= \frac{1}{\sqrt[p]{C_{n,\al,p}}}\sup_{k}\|u_k\|.\]
Then $(A_p)$ and fractional Hardy inequality \eqref{s1}, implies that
\begin{align}\label{h1}
\int_{\Om_1} V_1 |u_k|^p dx \leq \e \int_{\Om_1} \frac{|u_k|^p}{|x|^{p\al}} dx \leq \e c^p,
\end{align}
and similarly,
\begin{align}\label{h2}
\int_{\Om_1} V_1 |u|^p dx \leq  \e c^p.
\end{align}
By compactness, there is a finite covering of $\overline{\Om}_2$ by closed balls $B[x_1,r_1],\cdots,B[x_l,r_l]$ such that for $1\leq j\leq l,$
\begin{align}\label{h3}
|x-x_j|\leq r_j\quad \mbox{implies}\quad |x-x_j|^{p\al} V_2(x) \leq \e.
\end{align}
Then there exists $r= \min\{r_1,\cdots,r_l\}>0$ such that for $1\leq j\leq l$,
\[|x-x_j|\leq r \quad \mbox{implies}\quad |x-x_j|^{p\al} V_2(x) \leq \frac\e l.\]
Let us define $K:=\ds \cup_{j=1}^{l} B[x_j,r]$, then by the fractional Hardy inequality,
\begin{align}\label{h4}
\int_{K} V_2 |u_k|^p dx \leq \e c^p,\quad \int_{K} V_2 |u|^p dx \leq \e c^p.
\end{align}
It follows from \eqref{h3} that $V_2\in L^{\infty}(\Om_2\setminus K).$ Since $\Om_2\setminus K$ is bounded, $V_2\in L^{\frac{n}{p\al}}(\Om_2\setminus K)$ so that by step $1$,
\begin{align}\label{h5}
\int_{\Om_2\setminus K} V_2 |u_k|^p dx \ra \int_{\Om_2\setminus K} V_2 |u|^p dx.
\end{align}
We deduce from \eqref{h1}, \eqref{h2}, \eqref{h4} and \eqref{h5} that
\[\int_{\Om} V_2 |u_k|^p dx \ra \int_{\Om} V_2 |u|^p dx.\]
Combing the step $1$ and step $2$, we get the required result.\QED
\end{proof}

\begin{Theorem}\label{m1}
The eigenvalue $\la_1$ is attained.
\end{Theorem}

\begin{proof}
Let $\{u_k\}$ be a sequence in $M:=\{u\in X_0: \int_{\Om} V|u|^p dx=1\}$ such that $J(u_k)\ra \la_1$. Then by fractional Hardy inequality, $\{u_k\}$ is a bounded sequence in $X_0$ and $X_0$ is reflexive. Then there exists a subsequence $\{u_k\}$ of $\{u_k\}$ such that $u_k\rightharpoonup u$ weakly in $X_0$ and $u_k\ra u$ a.e in $\Om$. Since $J(u)\geq 0$ for $0\leq \mu<C_{n,\al,p}$. So it is bounded below on $M$. By Ekeland variational principle, we can find a sequence $\{v_k\}$ such that
\begin{align}
J(v_k)\leq J(u_k),\;\quad \|u_k-v_k\|\leq \frac{1}{k},\label{feq15}\\
J(v_k)\leq J(u)+\frac{1}{k}\|v_k-u\|\;\mbox{for all}\;u\in M. \label{feq020}
\end{align}
\noi Clearly $J(v_k)$ is a bounded, follows from
$\eqref{feq15}$. So we only need to prove that
$\|J^{\prime}(v_k)\|_{*}\ra 0$. Let $k>\frac{1}{\de}$
and take $w\in X_0$ tangent to $M$ at $v_k$ i.e
\[\int_{\Om}V |v_k|^{p-2} v_k w dx =0,\]

\noi and for $t\in \mb R$, define
\[u_t:= \frac{v_k+tw}{(\int_{\Om} V|v_k+tw|^p dx)^{\frac1p}}.\]
Then $u_t\in M$ and replacing $u$ by
$u_t$ in \eqref{feq020}, we get
\begin{align*}
J(v_k)&\leq
J(u_t)+\frac{1}{k}\|u_t-v_k\|.
\end{align*}
Define $r(t): =(\int_{\Om} V |v_k + tw|^p dx)^{\frac1p}$, we have
\begin{align*}
\frac{J(v_k)-J(v_k+tw)}{t}&\leq
\frac{J(u_t)+\frac{1}{k}\|u_t-v_k\|-J(v_k+tw)}{t}\\
&=\frac{1}{k\; t\;
r(t)}\|v_k(1-r(t)+tw)\|+\frac{1}{t}\left(\frac{1}{r(t)^p}-1\right)J(v_k+tw)
\end{align*}
As \[\frac{d}{dt}r(t)^p\ds|_{t=0}= p \int_{\Om}V |v_k|^{p-2} v_k
w=0,\] we obtain
\[\frac{r(t)^p-1}{t}\ra 0\;\mbox{as}\; t\ra 0,\]
and then
\[\frac{1-r(t)}{t}\ra 0\; \mbox{as}\; t\ra 0.\]
So, we have
\[\frac{J(v_k)-J(v_k+tw)}{t}\leq \frac{1}{k}\|w\|\; \mbox{as}\; t\ra 0,\]
that means
\begin{align}\label{feq16}
|\ld{J^{\prime}}(v_k),w\rd|\leq \frac{1}{k}\|w\|.
\end{align}

\noi Since $w$ is arbitrary in $X_0$, we choose $a_k$ such that
$\int_{\Om}V |v_k|^{p-2} v_k(w-a_k v_k) dx =0$. Replacing $w$ by
$w-a_k v_k$ in \eqref{feq16}, we have
\[\left|\ld{J^{\prime}}(v_k),w\rd - a_k \ld{J^{\prime}}(v_k), v_k\rd\right|\leq \frac{1}{k}\|w-a_k v_k\|,\]
since $\|a_k v_k\|\leq C\|w\|$, we get
\[\left|\ld{J^{\prime}}(v_k),w \rd - t_k \int_{\Om}|v_k|^{p-2}v_k w dx \right|\leq \frac{C}{k}\|w\|\]
where $t_k=\ld{J^{\prime}}(v_k), v_k \rd$. Hence, on $M$,
\begin{align}\label{k1}
\|J^{\prime}(v_k)\|_*\ra 0\;\mbox{as}\; k\ra\infty.
\end{align}
Thus we obtain that, $\{v_k\}$ is a Palais-Smale sequence. By the boundedness of the sequence $\{v_k\}$, we have that up to a subsequence $v_k\rightharpoonup v$ weakly in $X_0$. So, by the weak convergence of $v_k$ we have $\ld J^{\prime}(u),(u_k-u)\rd \ra \infty$.
and by \eqref{k1} $\ld J^{\prime}(v_k),v_k-v\rd \ra 0$ as $k\ra \infty$.
{\small\begin{align}\label{feq08}
&\left|\int_{Q}\frac{|v_k(x)-v_k(y)|^{p-2}(v_k(x)-v_k(y))((v_k- v)(x)-(v_k-v)(y))}{|x-y|^{n+p\al}}dxdy-\mu\int_{\Om} \frac{|v_k|^{p-2} v_k (v_k-v)}{|x|^{p\al}}dx\right|\notag\\
&\quad\leq O(\e_k)+ s\|v_{k}^{+}\|_{p}^{p-1}
\|v_k-v\|_{p}+|t_k|\|v_k\|_{p}^{p-1}
\|v_k-v\|_{p} \lra 0\;\mbox{as}\; k\ra\infty.
\end{align}}
By Brezis-Lieb Lemma \ref{bl1}, we have
\[\int_{Q}\frac{|(v_k-v)(x)-(v_k-v)(y)|^p}{|x-y|^{n+p\al}} dxdy= \int_{Q}\frac{|v_k(x)-v_k(y)|^p}{|x-y|^{n+p\al}} dxdy -\int_{Q}\frac{|v(x)-v(y)|^p}{|x-y|^{n+p\al}} dxdy+ o(1)\]
\[\left\|\frac{(v_k-v)}{|x|^\al}\right\|_{p}^p= \left\|\frac{v_k}{|x|^\al}\right\|_{p}^p- \left\|\frac{v}{|x|^\al}\right\|_{p}^p+ o(1)\]
Using above relation, we obtain from \eqref{s1} and \eqref{feq08},
{\small\begin{align*}
o(1)&=\int_{Q}\frac{|v_k(x)-v_k(y)|^{p-2}(v_k(x)-v_k(y))((v_k- v)(x)-(v_k-v)(y))}{|x-y|^{n+p\al}}dxdy-\mu\int_{\Om} \frac{|v_k|^{p-2} v_k (v_k-v)}{|x|^{p\al}}\\
&\geq \left(1-\frac{\mu}{C(n,p,\al)}\right)\int_{Q}\frac{|(v_k- v)(x)-(v_k-v)(y)|^p}{|x-y|^{n+p\al}}dxdy
\end{align*}}
Hence $v_k\ra v$ strongly in $X_0$. Thus as $k\ra \infty$ we obtain
\begin{align}
(-\De)^s v  -\mu \frac{|v|^{p-2}v}{|x|^{p\al}} = \la V(x) |v|^{p-2} v \; \text{in}\;
\mc D^{\prime}(\Om).
\end{align}
Observe that
\[\int_{\Om} V^- |v_k|^p dx =\int_{\Om} V^+ |v_k|^p dx -1 \ra \int_{\Om} V^+ |v|^p dx -1\]
as $k\ra\infty$. Now using Fatous Lemma, we can conclude that $v\not\equiv 0$. \QED
\end{proof}



\begin{Theorem}
$\la_1$ is simple in the sense that eigenfunctions associated to it are merely a constant multiple of each other.
\end{Theorem}

\begin{proof}
Let $\phi_1$ and $u$ are two eigenfunctions corresponding to the eigenvalue $\la_1$. Let $\{\psi_k\}$ be a sequence of functions such that $\psi_k \in C_{c}^{\infty}(\mb R^n)$ such that $\psi_k(x)=0$ in $\mb R^n\setminus\Om$, $\psi_k \geq 0$, $\psi_k \ra \phi_1$ in $X_0$ and convergent a.e. in $\Om$. Then we have
\begin{align}\label{k2}
0=& \int_{Q} \frac{|\phi_1(x)-\phi_1(y)|^{p}}{|x-y|^{n+p\al}} dxdy -\mu \int_{\Om} \frac{|\phi_1|^p}{|x|^{p\al}}dx - \la_1 \int_{\Om}V(x)|\phi_1|^p dx\notag\\
=&\lim_{k\ra\infty} \int_{Q} \frac{|\psi_k(x)-\psi_k(y)|^{p}}{|x-y|^{n+p\al}} dxdy -\mu \int_{\Om} \frac{|\psi_k|^p}{|x|^{p\al}}dx - \la_1 \int_{\Om}V(x)|\psi_k|^p dx.
\end{align}
Let $w_k:= \frac{\psi_{k}^p}{(u+\frac{1}{k})^{p-1}}$. Then we show that $w_k\in X_0$. Now,
{\small\begin{align*}
|w_k(x)-&w_k(y)|= \left|\frac{\psi_{k}^{p}(x) -\psi_{k}^{p}(y)}{(u+\frac{1}{k})^{p-1}(x)}+ \frac{\psi_{k}^{p}(y)((u+\frac{1}{k})^{p-1}(x)-(u+\frac{1}{k})^{p-1}(y))}{(u+\frac{1}{k})^{p-1}(x)(u+\frac{1}{k})^{p-1}(y)}\right|\\
&\leq k^{p-1}|\psi_{k}^{p}(x)-\psi_{k}^{p}(y)|+ \|\psi_{k}\|_{\infty}^{p}\frac{|(u+\frac{1}{k})^{p-1}(x)-(u+\frac{1}{k})^{p-1}(y)|}{(u+\frac{1}{k})^{p-1}(x)(u+\frac{1}{k})^{p-1}(y)}\\
&\leq k^{p-1}p (\psi_{k}^{p-1}(x) + \psi_{k}^{p-1}(y))|\psi_{k}(y)-\psi_{k}(x)|\\
&+\|\psi_{k}\|_{\infty}^{p}(p-1)\frac{|(u+\frac{1}{k})^{p-2}(x)+(u+\frac{1}{k})^{p-2}(y)|}{
(u+\frac{1}{k})^{p-1}(x)(u+\frac{1}{k})^{p-1}(y)}\left|\left(u+\frac{1}{k}\right)(x)-\left(u+\frac{1}{k}\right)(y)\right|\\
&\leq 2pk^{p-1} \|\psi_{k}\|_{\infty}^{p-1}  |\psi_{k}(x)-\psi_{k}(y)|\\
&\;+\|\psi_{k}\|_{\infty}^{p}(p-1)\left(\frac{1}{(u+\frac{1}{k})(x)(u+\frac{1}{k})^{p-1}(y)}+\frac{1}{(u+\frac{1}{k})^{p-1}(x)(u+\frac{1}{k})(y)}\right)|u(x)-u(y)|\\
&\leq C(k,p,\|\psi_{k}\|_{\infty}) (|\psi_{k}(x)-\psi_{k}(y)|+|u(x)-u(y)|)
\end{align*}}
for all $(x,y)\in \mb R^n\times \mb R^n$. Hence, $w_k\in X_0$ for all $k\in \mb N$, as $\psi_k$, $u\in X_0$.
Now, testing the equation satisfied by $u$ with $w_k$, we obtain
\begin{align}\label{k3}
\int_{Q} &\frac{|u(x)-u(y)|^{p-2}(u(x)-u(y))\left(\frac{\psi_{k}^p}{(u+\frac{1}{k})^{p-1}}(x)-\frac{\psi_{k}^p}{(u+\frac{1}{k})^{p-1}}(y)\right)}{|x-y|^{n+p\al}} dxdy\notag\\
&\quad= \int_{\Om} \left(\la_1 V(x)+\frac{\mu}{|x|^{p\al}}\right)\psi_{k}^p \left(\frac{u}{u+\frac{1}{k}}\right)^{p-1}dx
\end{align}
Now by equations \eqref{k2} and \eqref{k3}, we obtain
\begin{align*}
0 &= \lim_{k\ra\infty} \int_{Q} \frac{|\psi_k(x)-\psi_k(y)|^{p}}{|x-y|^{n+p\al}} dxdy \\
&\;\quad-\int_{Q} \frac{|u(x)-u(y)|^{p-2}(u(x)-u(y))\left(\frac{\psi_{k}^p}{(u+\frac{1}{k})^{p-1}}(x)-\frac{\psi_{k}^p}{(u+\frac{1}{k})^{p-1}}(y)\right)}{|x-y|^{n+p\al}} dxdy\\
&= \lim_{k\ra\infty} \int_{Q} L(\psi_k, u) \geq \int_{Q} L(\phi_1, u) \geq 0,
\end{align*}
by Fatous Lemma. Therefore by discrete Picone identity \eqref{dp1}, we have $\phi_1= l u$ a.e for some constant $l$. Hence $\la_1$ is simple.\QED
\end{proof}

\begin{Theorem}\label{pt3}
Eigenfunctions corresponding to other eigenvalues changes sign.
\end{Theorem}

\begin{proof}
Let $\phi_1$ and $u$ be the eigenfunctions corresponding to $\la_1$ and $\la$ respectively. Then $\phi_1$ and $u$ satisfies
\begin{align}
-2 \int_{\mb R^n}\frac{|\phi_1(x)- \phi_1(y)|^{p-2} (\phi_1(y)-\phi_1(x))}{|x-y|^{n+p\al}} dy -\mu \frac{|\phi_1|^{p-2}\phi_1}{|x|^{p\al}} &= \la_1 V(x) \phi_{1}^{p-1} \; \text{in}\;\mc D^{\prime}(\Om)\label{l1} \\
 -2 \int_{\mb R^n}\frac{|u(x)- u(y)|^{p-2} (u(y)-u(x))}{|x-y|^{n+p\al}} dy -\mu \frac{|u|^{p-2}u}{|x|^{p\al}} &= \la V(x)|u|^{p-2} u \; \text{in}\; \mc D^{\prime}(\Om) \label{l2}
\end{align}
 respectively. Suppose $u$ does not changes the sign. Then we may assume $u\geq 0$. Let $\{\psi_k\}$ be a sequence in $C_{c}^{\infty}(\mb R^n)$ such that $\psi_k=0$ in $\mb R^n\setminus \Om$, $\psi_k \ra \phi_1$ as $k\ra \infty$. Now we consider  the test functions $w_1=\phi_1$ , $w_2=\frac{\psi_{k}^p}{(u+\frac{1}{k})^{p-1}}$. Then $w_1$, $w_2\in X_0$. Taking $w_1$ and $w_2$ as test functions in \eqref{l1} and \eqref{l2} respectively, we obtain
 \begin{align}
 \int_{Q} \frac{|\phi_1(x)-\phi_1(y)|^{p}}{|x-y|^{n+p\al}} dxdy &-\mu \int_{\Om} \frac{|\phi_1|^p}{|x|^{p\al}}dx = \la_1 \int_{\Om}V(x)|\phi_1|^p dx\label{l3}\\
\int_{Q} \frac{|u(x)-u(y)|^{p-2}(u(x)-u(y))}{|x-y|^{n+p\al}}&\left(\frac{\psi_{k}^p}{(u+\frac{1}{k})^{p-1}}(x)-\frac{\psi_{k}^p}{(u+\frac{1}{k})^{p-1}}(y)\right) dxdy\notag\\
 -\mu \int_{\Om} \frac{|u|^{p-2}u\frac{\psi_{k}^p}{(u+\frac{1}{k})^{p-1}}}{|x|^{p\al}}dx
& = \la_1 \int_{\Om}V(x)|u|^{p-2}u \frac{\psi_{k}^p}{(u+\frac{1}{k})^{p-1}}dx.\notag
\end{align}
Since $L(\psi_k, u+\frac{1}{k})\geq 0$, we have
\begin{align}\label{l4}
\int_{Q}\frac{|\psi_k(x)-\psi_k(y)|^p}{|x-y|^{n+p\al}} dxdy - \int_{\Om} \left(\la V(x)+ \frac{\mu}{|x|^{p\al}}\right) \psi_{k}^{p} \left(\frac{u}{u+\frac{1}{k}}\right)^{p-1}\geq 0.
\end{align}
Subtracting \eqref{l3} from \eqref{l4} and taking limit as $k\ra \infty$, we obtain
\[(\la-\la_1) \int_{\Om}V(x) \phi_{1}^p \leq 0,\]
which gives a contradiction to fact that $\la>\la_1$.\QED
\end{proof}
\begin{Theorem}
Let $V_1$, $V_2$ be two weights and assume that $V_1\leq V_2$ a.e. and $|\{x\in \Om: V_1(x)<V_2(x)\}|\ne 0$. Then $\la_1(V_2) < \la_1(V_1)$.
\end{Theorem}

\begin{proof}
Let $u$ be an eigenfunction associated to $\la_1(V_1)$. Since
\[0<\la_1(V_1)^{-1}\left(\int_{Q} \frac{|u(x)-u(y)|^p}{|x-y|^{n+p\al}} dxdy- \mu \int_{\Om} \frac{|u|^{p}}{|x|^{p\al}}dx\right) =\int_{\Om} V_1|u|^p dx \leq \int_{\Om} V_2 |u|^p dx,\]
we use $\frac{u}{(\int_{\Om} V_2 |u|^p dx)^{\frac1p}}$ as an admissible function in infimum of \eqref{i1} for $\la_1(V_2)$. We have
\[\la_1(V_2) \leq \frac{\int_{Q} \frac{|u(x)-u(y)|^p}{|x-y|^{n+p\al}} dxdy- \mu \int_{\Om} \frac{|u|^{p}}{|x|^{p\al}}dx}{\int_{\Om} V_2|u|^pdx}\leq \frac{\int_{Q} \frac{|u(x)-u(y)|^p}{|x-y|^{n+p\al}} dxdy- \mu \int_{\Om} \frac{|u|^{p}}{|x|^{p\al}}}{\int_{\Om} V_1|u|^pdx} =\la_1(V_1).\]
Thus $\la_1(V_2)\leq \la_1(V_1)$. Suppose the equality holds if and only if $\int_{\Om} V_1 |u|^p dx=\int_{\Om} V_2 |u|^p dx$. This last identity implies that $V_1\equiv V_2$, which contradicts our hypothesis. Hence $\la_1(V_2) < \la_1(V_1)$.\QED
\end{proof}

\begin{Theorem}
Let $\Om_1$ be a proper open subset of a domain $\Om_2\subset \mb R^n$. Then $\la_1(\Om_2)<\la_1(\Om_1)$.
\end{Theorem}

\begin{proof}
Let $u\in X_0(\Om_1)$ be an eigenfunction associated to $\la_1(\Om_1)$ and put $\tilde{u}$ the function  obtained by extending $u$ by $0$ in $\Om_2\setminus\Om_1$. Then $\tilde{u}\in X_0(\Om_2)$ and $\int_{\Om_2} V |\tilde{u}|^p dx =\int_{\Om_1} V |\tilde{u}|^p dx >0$. Using $\frac{\tilde{u}}{(\int_{\Om} V_2|\tilde{u}|^p)^{\frac1p}}$ as an admissible function for $\la_1(\Om_2)$, we obtain,
{\small\[\la_1(\Om_2)\leq \frac{\int_{Q|_{\Om_2}}\frac{|\tilde{u}(x)-\tilde{u}(y)|^p}{|x-y|^{n+p\al}} dxdy - \mu \int_{\Om} \frac{|\tilde{u}|^{p}}{|x|^{p\al}}dx}{\int_{\Om} V_2|\tilde{u}|^p dx}\leq \frac{\int_{Q|_{\Om_1}}\frac{|u(x)- u(y)|^p}{|x-y|^{n+p\al}} dxdy- \mu \int_{\Om} \frac{|u|^{p}}{|x|^{p\al}}dx}{\int_{\Om} V_1|u|^p dx} = \la_1(\Om_1),\]}
where $Q|_{\Om_i}=\mb R^{2n}\setminus \mc C\Om_i\times \mc C \Om_i$ and $\mc C \Om_i= \mb R^n\setminus \Om_i$. The equality hold only if $\tilde{u}$ is an eigenfunction associated to $\la_1(\Om_2)$ but this is impossible because $|\tilde{u}=0|>0$ is a contradiction.\QED
\end{proof}

\section{Examples and counterexamples:}
If $\Om =\mb R^n$, then we have the following results:
\begin{Theorem}
If $|x|^{p\al} V(x)\ra \infty$ as $|x|\ra\infty$ or $|x-z|^{p\al} V(x)\ra \infty$ as $x\ra z$ for some $z$, then the infimum in \eqref{p2} is $0$ (and is not achieved).
\end{Theorem}

\begin{proof}
We only consider the case of  $|x|^{p\al} V(x)\ra \infty$ as $x\ra 0$, the other cases being similar. Let $u\in C_{c}^{\infty}(\mb R^n)$ and set $u_{r}(x)= u(x/r)$. Then
\begin{align}
\frac{\int_{\mb R^{2n}} \frac{|u_{r}(x)-u_{r}(y)|^p}{ |x-y|^{n+p\al}}dxdy - \int_{\mb R^n} \frac{|u_{r}(x)|^{p}}{|x|^{p\al}}dx}{\int_{\mb R^n} V(x) |u_{r}(x)|^p dx}=\frac{\int_{\mb R^{2n}} \frac{|u(x)-u(y)|^p}{|x-y|^{n+p\al}}dxdy- \int_{\mb R^n} \frac{|u(x)|^{p}}{|x|^{p\al}}dx}{\int_{\mb R^n}|rx|^{p\al} V(rx) \frac{|u(x)|^p}{|x|^{p\al}} dx}.
\end{align}
Since $u$ has compact support and $\frac{|u|^p}{|x|^{p\al}}\in L^{1}(\mb R^n)$, it follows easily that the right hand side above tends to $0$ as $r\ra 0$.
In case of $|x|\ra\infty$, the function $u\in C_{c}^{\infty}(\mb R^n)$ should be chosen so that $0\not\in$ supp $u$.\QED
\end{proof}
\noi Now we give some example of  potential $V$:

\begin{Example}
 Define
\[W_1(x)= \frac{1}{(1+|x|^{2\al})[\log(2+|x|^{2\al})]^{\frac{2\al}{n}}}\]
  and
\[W_2(x)= \frac{1}{|x|^{2\al}(1+|x|^{2\al})\left[\log\left(2+\frac{1}{|x|^{2\al}}\right)\right]^{\frac{2\al}{n}}}.\]
Then by Theorem \eqref{m2} for $p=2$ and Theorem \eqref{m3} for $p\ne2$, \eqref{p2} has infinitely many positive eigenvalue if $V= W_1$ or $W_2$ although $W_1$, $W_2$ are not in $L^{\frac{n}{2\al}}(\mb R^n)$. But we observe that $W_1\in L^{q}(\mb R^n)$ for all $q>\frac{n}{2\al}$, $W_2\in L^{q}(\mb R^n)$ for all $q\in(\frac{n}{4\al}, \frac{n}{2\al})$.

\noi If $W_3(x)=\frac{1}{1+|x|^{2\al}}$ and $W_4(x)= \frac{1}{|x|^{2\al}(1+|x|^{2\al})}$. Then $W_3$ and $W_4$ are not in $L^{\frac{n}{2\al}}(\mb R^n)$ and are in the same $L^q$-spaces as $W_1$ and $W_2$ respectively. Moreover, If we take $W_3(x),$ $W_4(x)= V_2$ as in assumption of $(A_p)$, we see that both the condition stated below, for $W_3$ and first condition for $W_3$,
\[\lim_{x\ra y, x\in\Om} |x-y|^{p\al} V_2(x)=0\; \mbox{for all}\; y\in \overline{\Om}, \;\mbox{and}\; \lim_{|x|\ra\infty, x\in\Om} |x|^{p\al} V_{2}(x)=0\]
 are not satisfied. So, we can't say anything about the existence of eigenvalue.\QED
\end{Example}
\section{On the Higher eigenvalue}
In this section, we show the existence of a sequence of eigenvalues $\{\la_k\}$ of \eqref{p2} such that $\la_k\ra\infty$ as $k\ra\infty$ using the Ljusternik-Schnirelman critical point theory on $C^{1}$ manifold.

 In order to obtain the higher eigenvalue of \eqref{p2} in the linear case $p=2$, we solve the following problem
{\small\begin{align*}
(P_k)&\mbox{ minimize} \int_{Q}\frac{|u(x)-u(y)|^2}{|x-y|^{n+2\al}} dxdy - \mu \int_{\Om}\frac{|u|^2}{|x|^{2\al}}dx, u\in X_0\\
\int_{Q}& \frac{(u(x)-u(y)(\phi_1(x)-\phi_1(y)))}{|x-y|^{n+p\al}} dxdy= \cdots=\int_{Q} \frac{(u(x)-u(y)(\phi_{k-1}(x)-\phi_{k-1}(y)))}{|x-y|^{n+p\al}} dxdy= 0,\\
& \int_{\Om} V u^2 dx=1, \;\mbox{where}\; \phi_j\; \mbox{is a solution of}\; (P_{j}), 1\leq j\leq k-1.
 \end{align*}}
{\small\begin{Theorem}\label{m2}
Under assumption $(A_2)$, for every $n>2\al$, problem $(P_{k})$ has a solution $\phi_k$. Moreover, $\phi_k$ is an eigenfunction of $(P_{k})$ corresponding to eigenvalue $\la_k: = \int_{Q}\frac{|\phi_{k}(x)-\phi_{k}(y)|^2}{|x-y|^{n+2\al}} dxdy - \mu \int_{\Om}\frac{\phi_{k}^2}{|x|^{2\al}}dx$  and $\la_k\ra\infty$ as $k\ra\infty$.
\end{Theorem}}

\begin{proof}
The existence of $\phi_k$ is proved as in Theorem \ref{m1}. An elementary argument in \cite{var} shows that $\phi_k$ is an eigenfunction of $(P_{k})$ corresponding to eigenvalue
\[\la_{k}: = \int_{Q}\frac{|\phi_{k}(x)-\phi_{k}(y)|^2}{|x-y|^{n+2\al}} dxdy - \mu \int_{\Om}\frac{\phi_{k}^2}{|x|^{2\al}}dx.\]
The sequence $f_k:= \phi_k/ {\sqrt{\la_{k}}}$ is orthonormal in $X_0$ so that $f_k\rightharpoonup 0$. Since
\[ \la_{k}^{-1}= \la_{k}^{-1}\left(\int_{Q}\frac{|f_{k}(x)-f_{k}(y)|^2}{|x-y|^{n+2\al}} dxdy - \mu \int_{\Om}\frac{f_{k}^2}{|x|^{2\al}}dx\right) = \int_{\Om} V f_{k}^2 dx,\]
Lemma \ref{v1} implies that
\begin{align*}
0\leq \lim_{k\ra\infty} \la_{k}^{-1}= \lim_{k\ra\infty} \int_{\Om} V f_{k}^2 dx\leq 0,
\end{align*}
which completes the proof of Lemma.\QED
\end{proof}
\noi Since the equation \eqref{p2} is nonlinear (unless $p=2$), it is not possible to obtain higher eigenvalues by the method of above Theorem \ref{m2}. For this, we will use the Ljusternik-Schnirelman critical point theory on $C^{1}$ manifold proved by \cite{st}. Let
\[J_{\mu}(u):=  \int_{Q}\frac{|u(x)-u(y)|^p}{|x-y|^{n+p\al}} dxdy -\mu \int_{\Om}\frac{|u|^p}{|x|^{p\al}} dx \;\mbox{and}\; \Psi(u):= \int_{\Om} V |u|^p
dx. \]
Since the set $\{u\in X_0 : \int_{\Om} V|u|^p dx=1 \}$ is a not a manifold in $X_0$ unless further assumptions are made on $V^-$. So we introduce a new space $X:= \{u\in X_0 :   \|u\|_{X}<\infty\},$ where
\[\|u\|^{p}_{X} := \int_{Q}\frac{|u(x)-u(y)|^p}{|x-y|^{n+p\al}} dxdy+ \int_{\Om} V^- |u|^p dx. \]
Then $M:=\{u\in X: \int_{\Om} V|u|^p dx=1\}$  is a $C^1-$manifold as a subset of the space $X$. Moreover critical points of $\phi|_{M}$ are eigenfunctions and corresponding critical values are eigenvalues of \eqref{p2}.  On this space, we define the functional
\[J_{\mu}(u)= \frac{ \int_{Q}\frac{|u(x)-u(y)|^p}{|x-y|^{n+p\al}} dxdy -\mu \int_{\Om}\frac{|u|^p}{|x|^{p\al}}dx}{\int_{\Om} V |u|^p
dx}.\]
Let $\tilde{J}_{\mu}$ denote the restriction of $J_{\mu}$ to $M$ and let $\Psi^{\pm}(u):= \int_{\Om} V^{\pm} |u|^p dx$. Then we prove the following Lemma:
\begin{Lemma}
If $V$ satisfies $(A_p)$, then the following holds:
\begin{enumerate}
\item[$(i)$] The Fr\'{e}chet derivatives of $\Psi^{+}$ is completely continuous as a mapping from $X$ to $X^*$.
\item[$(ii)$]  $\Psi^{+}(u)\leq c J_{\mu}(u)$ for some $c>0$ and for all $u\in X$, $0\leq u< C_{n,\al,p}$.
\end{enumerate}
\end{Lemma}

\begin{proof}
\begin{enumerate}
\item[$(i)$] Let $u_k\rightharpoonup u$. By the H\"{o}lder and the fractional Sobolev inequalities,
\begin{align*}
\int_{\Om} V_1 (|u_k|^{p-2} u_k -|u|^{p-2} u)v dx &\leq \left(\int_{\Om} V_1 ||u_k|^{p-2}u_k  -|u|^{p-2} u|^{\frac{p}{p-1}}dx\right)^{\frac{p-1}{p}} \left(\int_{\Om} V_1 |v|^{p}dx\right)^{\frac{1}{p}}\\
&\leq c_1 \|v\|_{X}\left(\int_{\Om} V_1 ||u_k|^{p-2}u_k  -|u|^{p-2} u|^{\frac{p}{p-1}}dx\right)^{\frac{p-1}{p}}.
\end{align*}
It is easy to see that $||u_k|^{p-2} u_k -|u|^{p-2} u|^{\frac{p}{p-1}}\rightharpoonup 0$ in $L^{\frac{n}{n-p\al}}(\Om)$ (indeed, otherwise there would exist a subsequence converging weakly to some $v\not= 0$ and a.e to $0$, a contradiction). Since $V_1\in L^{n/{p\al}}(\Om)$, the right hand side above tends to $0$ uniformly for $\|v\|_{X}\leq 1$. This shows the complete continuity of the $V_1-$part.

Using the notation of Lemma \ref{v1} and the H\"{o}lder, the fractional Hardy and the fractional Sobolev inequalities, we see that
\[\int_{\Om_1} V_2 (|u_k|^{p-2} u_k -|u|^{p-2} u)v dx \leq c_2 \e \|v\|_X (\|u_k\|^{p-1}_{X}+ \|u_k\|^{p-1}_{X}) \leq c_3 \e \|v\|_X.\]
Similarly
\[\int_{K} V_2 (|u_k|^{p-2} u_k -|u|^{p-2} u)v dx  \leq c_4 \e \|v\|_X,\]
where $c_i:s$ are independent of $\e$. Since $\Om_2\setminus K$ is bounded and $V_2 \in L^{\infty}(\Om_2\setminus K)$. It follows that $|u_k|^{p-2}u_k \rightharpoonup |u|^{p-2}u$ in $L^{\frac{p}{p-1}}(\Om_2\setminus K)$ and
\[\int_{\Om_2\setminus K} V_2 (|u_k|^{p-2} u_k -|u|^{p-2} u)v dx \ra 0.\]
\item[$(ii)$] By the H\"{o}lder and the fractional Sobolev inequalities,
\[\int_{\Om} V_1 |u|^{p} dx \leq c_5 \int_{\Om} \frac{|u(x)-u(y)|^{p}}{|x-y|^{n+p\al}} dxdy \leq C_5 J_{\mu}(u).\]
Fixing some $\e>0$ and using the H\"{o}lder, the fractional Hardy and the Sobolev inequality again, it is easy to see that
\[\int_{\Om_1} V_2 |u|^{p} dx \leq c_6 \int_{\Om} \frac{|u(x)-u(y)|^{p}}{|x-y|^{n+p\al}} dxdy \leq C_6 J_{\mu}(u).\]
and similar inequalities holds for $K$ and $\Om\setminus K$. Hence the conclusion now follows by recalling the definition of $\Psi^{+}$ and $J_{\mu}$. \QED
\end{enumerate}
\end{proof}
Let $t>0$ and let $A_{t}: X\ra X^*$ be the operator given by
{\small\[\langle A_{t}(u), \phi\rangle = \int_{Q}\frac{|u(x)-u(y)|^{p-2}(u(x)-u(y))(\phi(x)-\phi(y))}{|x-y|^{n+p\al}}dx dy-\mu \int_{\Om}\frac{|u|^{p-2}uv}{|x|^{p\al}}  + t\int_{\Om} V^- |u|^{p-2} u\phi .\]}
\begin{Lemma}\label{le1}
If $u_k\rightharpoonup u$ weakly in $X$ and $\langle A_{t}(u_k), u_k-u\rangle \ra 0,$ then $u_k\ra u$ strongly in $X$.
\end{Lemma}

\begin{proof}
Clearly, using the weak convergence of $u_k$ and given hypothesis, we have
\begin{align}\label{a1}
\langle A_{t}(u_k)- A_{t}(u), u_k-u\rangle \ra 0.
\end{align}
Now, using the H\"{o}lder inequality, we obtain
{\small\begin{align}
\int_{\Om}& V^-(|u_k|^{p-2} u_k -|u|^{p-2} u)(u_k-u) dx = \int_{\Om} V^- (|u_k|^{p}+|u|^p - |u_k|^{p-2}u_k u -|u|^{p-2} u u_k) dx\notag\\
&\geq \int_{\Om} V^- (|u_k|^{p}+|u|^p)- \left(\int_{\Om} V^- |u_k|^{p}\right)^{\frac{p-1}{p}} \left(\int_{\Om} V^-|u|^{p}\right)^{\frac{1}{p}}- \left(\int_{\Om} V^- |u|^{p}\right)^{\frac{p-1}{p}}\left(\int_{\Om} V^- |u_k|^{p}\right)^{\frac{1}{p}}\notag\\
&= \left[\left(\int_{\Om} V^- |u_k|^{p}\right)^{\frac{p-1}{p}}-\left(\int_{\Om} V^- |u|^{p}\right)^{\frac{p-1}{p}}\right]\left[\left(\int_{\Om} V^- |u_k|^{p}\right)^{\frac{1}{p}}-\left(\int_{\Om} V^- |u|^{p}\right)^{\frac{1}{p}}\right]\geq 0.\label{a2}
\end{align}}
Now, by Brezis-Lieb Lemma \eqref{bl1}, we have
\[\|u_k-u\|= \|u_k\|-\|u\|+ o(1)\]
\[\left\|\frac{u_k-u}{|x|^{\al}}\right\|_{p}^p= \left\|\frac{u_k}{|x|^{\al}}\right\|_{p}^p- \left\|\frac{u}{|x|^{\al}}\right\|_{p}^p+ o(1).\]
which implies that
\begin{align*}
o(1) &= \langle A_{t}(u_k)- A_{t}(u), u_k-u\rangle \\
&= \ld J^{\prime}_{\mu}(u_k)- J^{\prime}_{\mu}(u), (u_k-u)\rd + t\left(\int_{\Om} V^{-}(|u_k|^{p-2} u_k -|u|^{p-2} u)(u_k-u) dx\right)\\
&\geq \int_{Q}\frac{|(u_k-u)(x)-(u_k-u)(y)|^{p}}{|x-y|^{n+p\al}}dx dy-\mu \int_{\Om}\frac{|u_k-u|^{p}}{|x|^{p\al}} dx\\
&\geq \left(1-\frac{\mu}{c(n,p,\al)}\right)\|u_k-u_0\|+o(1),
\end{align*}
Thus $u_k\ra u$ strongly in $X_0$. Using this, \eqref{a1} and \eqref{a2}, we have
\[\int_{\Om} V^- |u_k|^{p} dx \ra \int_{\Om} V^- |u|^{p} dx.\]
Hence $u_k\ra u$ strongly in $X$.\QED
\end{proof}

We define $\la_k=\inf_{\ga(A)\geq k}\sup_{u\in A} J_{\mu}(u)$, where $A$ is a closed subset of $M$ such that $A=-A$ and $\ga(A)$ is the Krasnoselskii genus of $A$. Since $\{x\in \mb R^n; V(x)>0\}$ has positive measure, for each $k$ there is a set $A\subset M$ which is homeomorphic to the unit sphere $S^{k-1}\subset \mb R^k$ by an odd homeomorphism. Since $\ga(S^{k-1})=k,$ there exist sets of arbitrarily large genus and all $\la_k$ are all well-defined. Moreover, $\la_1=\inf_{u\in M} J_{\mu}(u)$. Hence $\la_1$ coincides with the first eigenvalue obtained in Theorem \ref{m1} and $\la_k\geq \la_1>0$ for all $k$.

As $\la_k$ is a critical value of $J_{\mu}|_M$, there exists a critical point $\phi_k$ with $J_{\mu}(\phi_k)= \la_k$. Hence $J_{\mu}^{\prime}(\phi_k)= \ba \psi^{\prime}(\phi_k)$, where $\ba$ is a Lagrange multipliers and is satisfied with $u=\phi_k$ and $\ba=\la$. Since $p J_{\mu}(\phi_k)= \langle J_{\mu}^{\prime}(\phi_k), \phi_k\rangle= \ba \langle \psi^{\prime}(\phi_k), \phi_k\rangle=p\ba$, we have $\ba=J_{\mu}(\phi_k) =\la_k$, so $\la_k$ is an eigenvalue and $\phi_k$ is corresponding eigenfunction.

\begin{Theorem}\label{m3}
Under assumption $(H_p)$, $J_{\mu}|M$ has a sequence of critical points $\phi_k$ with corresponding critical value $\la_k=\int_{Q} \frac{|u(x)-u(y)|^p}{|x-y|^{n+p\al}}dxdy$. Moreover, each $\phi_k$ is an eigenfunction of \eqref{p2}, $\la_k$ is an associated eigenvalue, and $\la_k\ra\infty$ as $k\ra\infty$.
\end{Theorem}

\begin{proof}
Let $\{u_k\}$ be a Palaise-Smale sequence. Then there exists $\eta_k\in \mb R^n$ such that
\begin{align}\label{a}
J_{\mu}^{\prime}(u_k) -\eta_k \Psi^{\prime}(u_k)\ra 0.
\end{align}
Since $J_{\mu}(u_k)$ is bounded, so is $\Psi^{+}(u_k)$ according to Lemma  and therefore
\begin{align}\label{b}
\Psi^{-}(u_k) = \Psi^{+}(u_{k})-1
\end{align}
is bounded. Hence $\|u_k\|_{X}^p= J_{\mu}(u_k)+\Psi^{-}(u_k)$
is bounded and we assume passing to a subsequence that $u_k \rightharpoonup u$ weakly in $X$. Since $(\Psi^{+})^{\prime}$ is completely continuous, $\Psi^{+}(u_k)\ra \Psi^{+}(u)$ and it follows from \eqref{b} that $u\ne 0$. By \eqref{a}
\begin{align}
p(J_{\mu}(u_k)-\eta_k)=\langle J^{\prime}_{\mu}(u_k), u_k \rangle -\eta_{k}\langle \Psi^{\prime}(u_k), u_k \rangle\ra 0.
\end{align}
Therefore $\eta_k$ is bounded and we assume that $\eta_k \ra \eta$. Moreover taking the limit in above, we obtain $0<J_{\mu}(u)\leq \eta$, so $\eta>0$. Now we may rewrite \ref{a}, as
\[A_{\eta_k}(u_k) - \eta_k (\Psi^{+})^{\prime}(u_k) \ra 0.\]
Since $A_{\eta_k}(u_k) - A_{\eta}(u_k) \ra 0$, as it  can be easily seen from the definition of $A_{\eta}$ and since $(\Psi^{+})^{\prime}(u_k)\ra (\Psi^{+})^{\prime}(u)$. It follows that $A_{\mu}(u_k)$ is strongly convergent. So, $\la A_{\eta}(u_k), u_k-u\ra 0$ and $u_k \ra u$ according to Lemma \ref{le1}. Thus we have shown that $J_{\mu}|_M$ satisfy the Palaise-Smale condition. It follows from our earlier  discussion that each $\la_k$ is  a critical value of $J_{\mu}|M$ and an eigenvalue of the problem . Since $\la_k\geq C \la_{k}^{0}$ are eigenvalues of $L_{0}$, we have $\la_k\ra \infty$ as $k\ra\infty$ see Proposition of 2.2 in \cite{as}.
\end{proof}



\begin{thebibliography}{21}
\footnotesize
\bibitem{ad} Adimurthi, {\it Hardy–Sobolev inequality in $H^1(\Om)$ and its applications}, Commun. Contemp. Math., 3 (2002) 409-434.

\bibitem{as} Adimurthi and K. Sandeep, {\it Existence and non-existence of the first eigenvalue of the perturbed Hardy–Sobolev operator,} Proc. Roy. Soc. Edinburgh Sect. A, 132 (5) (2002) 1021-1043.

\bibitem{da} D. Applebaum, {\it L\`{e}vy process-from probability to finance and quantum groups}, Notices Amer. Math. Soc., {51} (2004) 1336-1347.

\bibitem{am} S. Amghibech, {\it On the discrete version of Picone's identity}, Discrete App. Math., 156 (1) (2008) 1-10.

\bibitem{bl} H. Brezis and E. Lieb, {\it A relation between point convergence of functions and convergence of functionals}, Proc. AMS, 88 (1983) 486-490.


\bibitem{bps} L. Brasco, E. Parini and M. Squassina, {\it Stability of variational eigenvalues for the fractional $p$-Laplacian}, arXiv:1503.04182v1.

\bibitem{cs} L. Caffarelli and L. Silvestre, {\it An extension problem related to the fractional Laplacian}, Comm. in Partial Differential Equations. \textbf{32} (2007) 1245-1260.

\bibitem{hic} E. Di Nezza, G. Palatucci and E. Valdinoci, {\it Hitchhiker's guide
to the fractional Sobolev spaces}, Bull. Sci. Math., 136 (2012) 225-236.

\bibitem{pfr} L. D. Pezzo, J. F. Bonder and L. L. Rios, {\it An optimization problem for the first eigenvalue of the $p-$fractional Laplacian,} arXiv:1601.03019v1 [math.AP].

\bibitem{fp} G. Franzina and G. Palatucci, {\it Fractional $p$-eigenvalues}, Riv. Mat. Univ. Parma, available at \\
http://arxiv.org/pdf/1307.1789v1.pdf.

\bibitem{gm} A. Garroni and S. M\"{u}ller, {\it G-limit of a phase-field model of dislocations}, SIMA J. Math. Anal.. \textbf{36} (2005) 1943-1964.

\bibitem{as1} A. Iannizzotto and M. Squassina, {\it Weyl-type laws for fractional $p-$eigenvalue problems}, Asymptot. Anal., 88 (1) (2014) 233-245.

\bibitem{ll} E. Lindgren and P. Lindqvist, {\it Fractional eigenvalues}, Calc. Var.
Partial Differential Equations, 49 (2013) 795-826.

\bibitem{pq} L. M. Del Pezzo and Alexander Quaas, {\it Global bifurcation for fractional $p$-Laplacian and application},  arXiv:1412.4722v2.

\bibitem{fs}  Rupert L. Frank and Robert Seiringer, {\it Non-linear ground state representations and sharp Hardy inequlities}, Journal of Functional Analysis, 255 (2008) 3407-3430.

\bibitem{ssa} S. Goyal and K. Sreenadh, {\it Existence of multiple solutions of $p$-fractional Laplace operator with sign-changing weight function,}
 Adv. Nonlinear Anal., 4 (1) (2015) 37-58.

\bibitem{ssi} S. Goyal and K. Sreenadh, {\it The Nehari manifold for non-local elliptic operator with concave-convex nonlinearities and sign-changing weight functions}, Proc. Indian Acad. Sci. Math. Sci., 125 (4) (2015) 545-558.	


\bibitem{vs} V. Raghavendra and K. Sreenadh, {\it Strong resonance for Hardy-Sobolev operator}, Applicable analysis, 82 (3) (2003) 241-252.

\bibitem{sre} K. Sreenadh, {\it On the eigenvalue problem for the Hardy-Sobolev operator with indefinite weights}, Electronic Journal of differential equations, 33 (2002) 1-12.

\bibitem{ss} K. Sandeep and K. Sreenadh, {\it Asymptotic behaviour of the rst eigenfunction of a perturbed Hardy-Sobolev operator}, Nonlinear Analysis: Theory, Methods and applications, 53 (3) (2003) 545-563.

\bibitem{s2} K. Sreenadh, {\it On the second eigenvalue of a Hardy-Sobolev operator}, Electronic Journal of differential equations, 12 (2004) 1-9.

\bibitem{mp} R. Servadei and E. Valdinoci, {\it Mountain pass solutions for non-local elliptic operators,} J. Math. Anal. Appl., 389 (2) (2012) 887-898.

\bibitem{var} R. Servadei and E. Valdinoci, {\it Variational methods for non-local operators of elliptic type,}
 Discrete Contin. Dyn. Syst., 33 (5) (2013) 2105-2137.

\bibitem{sv} R. Servadei and E. Valdinoci, {\it Lewy-Stampacchia type estimates for
variational inequalities driven by non-local operators}, Rev. Mat. Iberoam., 29 (2013) 1091-1126.

\bibitem{bn1} R. Servadei and E. Valdinoci, {\it A Brezis Nirenberg result for non-local critical equations in low dimension,} Commun. Pure Appl.
    Anal., 12 (6) (2013) 2445-2464.

\bibitem{bn} R. Servadei and E. Valdinoci, {\it The Brezis Nirenberg result for the fractional Laplacian}, Trans. Amer. Math. Soc., 367 (2015) 67-102.





\bibitem{sw} A. Szulkin and M. Willem, {\it Eigenvalue problems with indefinite weights}, Stud. Math, 135 (2) (1999) 199-208.

\bibitem{s} A. Szulkin, {\it Ljusternik-Schnirelmann theory on $C^1-$manifolds,} Ann. Inst. H. Poincare Anal, Non Lineaire, 5 (1988) 119-139.

\bibitem{st} M. Struwe, {\it Variational Methods}, Springer, New york 2000.

\bibitem{wi} M. Willem, {\it Minimax methods}, Birkhauser, Boston, 1996.

\end{thebibliography}
 \end{document}